\documentclass[12pt]{amsart}
\usepackage{amscd,amssymb,amsthm,verbatim,slashed}
\newcommand{\C}{{\mathbb C}}

\newcommand{\ch}{\operatorname{ch}}

\newcommand{\dvol}{\operatorname{dvol}}

\newcommand{\HH}{\operatorname{H}}

\newcommand{\Id}{\operatorname{Id}}

\newcommand{\Pin}{\operatorname{Pin}}

\newcommand{\Q}{{\mathbb Q}}
\newcommand{\R}{{\mathbb R}}

\newcommand{\Ric}{\operatorname{Ric}}

\newcommand{\Spin}{\operatorname{Spin}}

\newcommand{\Z}{{\mathbb Z}}

\numberwithin{equation}{section}
\theoremstyle{plain}

\newtheorem{lemma}[equation]{Lemma}
\newtheorem{theorem}[equation]{Theorem}
\newtheorem{proposition}[equation]{Proposition}
\newtheorem{corollary}[equation]{Corollary}

\theoremstyle{remark}

\newtheorem{remark}[equation]{Remark}

\usepackage{graphicx}
\input{amssym.def}
\input{amssym}
\input{epsf}
\usepackage{a4wide}

\begin{document}

\title[Index theory for scalar curvature on manifolds with boundary]
      {Index theory for scalar curvature on manifolds with boundary}

\author{John Lott}
\address{Department of Mathematics\\
University of California, Berkeley\\
Berkeley, CA  94720-3840\\
USA} \email{lott@berkeley.edu}

\thanks{Research partially supported by NSF grant
DMS-1810700}
\date{October 1, 2020}
\begin{abstract}
We extend results of Llarull and Goette-Semmelmann to manifolds with boundary.
\end{abstract}

\maketitle

\section{Introduction} \label{sect1}

Llarull showed that the sphere has no Riemannian metric
that is greater than the standard round metric and also has a larger
scalar curvature \cite{Llarull (1996),Llarull (1998)}.
Goette and Semmelmann gave an
extension of Llarull's result in which the sphere is replaced by a
manifold with nonnegative curvature operator
\cite{Goette-Semmelmann (2002)}.

In \cite{Gromov (2019)},
Gromov discussed questions about scalar curvature, including an
extension of these results to manifolds with boundary. We first give an
extension of the Goette-Semmelmann result. Let $R$ denote scalar
curvature and let $H$ denote mean curvature.

\begin{theorem} \label{1.1}
  Let $N$ and $M$ be compact connected
  even dimensional Riemannian manifolds with boundary.
  Let $f \: : \: N \rightarrow M$ be a smooth spin map and let
  $\partial f \: : \: \partial N \rightarrow \partial M$
  denote the restriction to the boundary.
Suppose that
\begin{itemize}
\item $f$ is $\Lambda^2$-nonincreasing and $\partial f$ is
  distance-nonincreasing,
\item $M$ has nonnegative curvature operator and
    $\partial M$ has nonnegative second fundamental form,
    \item 
    $R_N \ge f^* R_M$ and
      $H_{\partial N} \ge (\partial f)^* H_{\partial M}$,
      \item 
        $M$ has nonzero Euler characteristic and
        \item 
          $\int_N \widehat{A}(N) f^*[M, \partial M] \neq 0$.
\end{itemize}
Then $R_N = f^* R_M$ and $H_{\partial N} = (\partial f)^* H_{\partial M}$.

Furthermore,
  \begin{itemize}
\item    If $0 < \Ric_M < \frac12 R_M g_M$ then $f$ is a Riemannian submersion.
\item If $\Ric_M > 0$ and $f$ is distance-nonincreasing then $f$ is a
  Riemannian submersion.
\item If $M$ is flat then $N$ is Ricci-flat.
  \end{itemize}
\end{theorem}

Here the map $f$ is spin if $TN \oplus f^* TM$ admits a spin structure.
The class $[M, \partial M] \in \HH^m(M, \partial M; o_M)$ is the fundamental
class in cohomology twisted by the real orientation line bundle $o_M$,
so $f^*[M, \partial M] \in \HH^m(N, \partial N; o_N)$. Also
$\widehat{A}(N) \in \HH^*(N; \Q)$  is the $\widehat{A}$-class and
$\int_N$ denotes pairing with the fundamental class in
$\HH_n(N, \partial N; o_N)$. The quantity
$\int_N \widehat{A}(N) f^*[M, \partial M]$ is called the
$\widehat{A}$-degree in \cite{Gromov-Lawson (1983)}.
When $\partial N = \partial M = \emptyset$, Theorem \ref{1.1} recovers the
Goette-Semmelmann result.

When specialized to the case $\dim(N) = \dim(M)$, we obtain the
following extension of Llarull's result from \cite{Llarull (1998)}.

\begin{corollary} \label{1.2}
  Let $N$ and $M$ be compact connected Riemannian manifolds with boundary
  of the same even dimension.
  Let $f \: : \: N \rightarrow M$ be a smooth spin map and let
  $\partial f \: : \: \partial N \rightarrow \partial M$ denote the
  restriction to the boundary.
  Suppose that
\begin{itemize}
  \item $f$ is $\Lambda^2$-nonincreasing and $\partial f$ is
  distance-nonincreasing,
  \item $M$ has nonnegative curvature operator and
    $\partial M$ has nonnegative second fundamental form,
    \item 
    $R_N \ge f^* R_M$ and
      $H_{\partial N} \ge (\partial f)^* H_{\partial M}$,
      \item 
        $M$ has nonzero Euler characteristic and
        \item 
          $f$ has nonzero degree.
\end{itemize}
Then $R_N = f^* R_M$ and $H_{\partial N} = (\partial f)^* H_{\partial M}$.

Furthermore,
  \begin{itemize}
  \item    If $0 < \Ric_M < \frac12 R_M g_M$ then $f$ is a Riemannian covering
    map.
\item If $\Ric_M > 0$ and $f$ is distance-nonincreasing then $f$ is a
  Riemannian covering map.
\item If $M$ is flat then $N$ is Ricci-flat.
  \end{itemize}
\end{corollary}

Gromov proved the first part of Corollary \ref{1.2} when
$M$ is a ball in Euclidean space
\cite[Section 2]{Gromov (2018)},
\cite[Section 3.6]{Gromov (2019)}.
(His interest in this case came from an application to hypersurfaces in
Euclidean space.) Gromov's proof used a geometric doubling
of $N$ and a limiting procedure, to apply the Goette-Semmelmann result.
We apply index theory directly to a Dirac-type operator on $N$, with
local boundary conditions. In general there are topological obstructions
to the existence of local boundary conditions for Dirac-type operators,
but in our case the obstruction vanishes.  The proof of Theorem \ref{1.1}
effectively uses an analytic doubling argument.

Gromov asked what happens to the Dirac-type operator in his argument when
one passes to the limit.  Presumably one recovers the operator that we use.

In addition to local boundary conditions, one can consider nonlocal
Atiyah-Patodi-Singer boundary conditions \cite{Atiyah-Patodi-Singer (1975)}.
This leads to the following result.

\begin{theorem} \label{1.3}
  Let $N$ and $M$ be compact connected Riemannian manifolds with boundary
  of the same even dimension.
  Let $f \: : \: N \rightarrow M$ be a smooth spin map and let
  $\partial f \: : \: \partial N \rightarrow \partial M$ denote the
  restriction to the boundary.
  Suppose that
\begin{itemize}
\item $f$ is $\Lambda^2$-nonincreasing,
\item $\partial f$ is an isometry and preserves the
  second fundamental forms,
    \item $M$ has nonnegative curvature operator and
    $\partial M$ has vanishing mean curvature, and
    \item 
    $R_N \ge f^* R_M$.
\end{itemize}
Then $R_N = f^* R_M$ and $H_{\partial N} = (\partial f)^* H_{\partial M}$.

Furthermore,
  \begin{itemize}
  \item    If $0 < \Ric_M < \frac12 R_M g_M$ then $f$ is an isometry.
    \item If $\Ric_M > 0$ and $f$ is distance-nonincreasing then $f$ is an
isometry.
\item If $M$ is flat then $N$ is Ricci-flat.
  \end{itemize}
\end{theorem}

Comparing Corollary \ref{1.2} and Theorem \ref{1.3}, one difference is that
Corollary \ref{1.2} assumes nonnegativity of the second fundamental form of $M$,
while Theorem \ref{1.3} assumes vanishing of its trace.
In Corollary \ref{1.2}
the boundary map is assumed to be distance nonincreasing,
while in Theorem \ref{1.3} it is actually an isometry and it
preserves the second fundamental form.  

I thank Dan Freed and Chao Li for correspondence.

\section{Proof of Theorem \ref{1.1}}

\subsection{Bochner-type argument}

For simplicity, we assume that $N$ and $M$ are spin; the general case is
similar.
Put $E = S_N \otimes f^* S_M$, a Clifford module on $N$.
(This Clifford module exists in the general case.)
We take the inner product $\langle \cdot, \cdot \rangle$ on $E$
to be $\C$-linear in the
second slot and $\C$-antilinear in the first slot.

Let $\omega^\alpha_{\: \: \beta \gamma}$ be the connection $1$-forms
with respect to a local orthonormal framing
$\{e_\alpha\}_{\alpha = 1}^n$ on $N$. Let
$\widehat{\omega}^a_{\: \: b \gamma}$ be the pullbacks under $f$ of
connection $1$-forms with respect to a local orthonormal framing
$\{e_a\}_{a = 1}^m$ of $M$.

  Let $\{\gamma^\alpha\}_{\alpha = 1}^n$ be generators of the Clifford algebra
  on $\R^n$,
  satisfying $\gamma^\alpha \gamma^\beta + \gamma^\beta \gamma^\alpha =
  2 \delta^{\alpha \beta}$.
  Let $\{\widehat{\gamma}^a \}_{a = 1}^m$ be the analogous generators of the
  Clifford algebra on
  $\R^m$.
The covariant derivative on $E$ has the local form
\begin{equation} \label{2.1}
\nabla^N_\sigma = e_\sigma + \frac{1}{8} \omega_{\alpha \beta \sigma}
[\gamma^\alpha, \gamma^\beta] +
\frac{1}{8} \widehat{\omega}_{ab \sigma}
     [\widehat{\gamma}^a, \widehat{\gamma}^b].
     \end{equation}
The Dirac operator on $C^\infty(N; E)$ is
$D^N = - i \sum_{\sigma=1}^n \gamma^\sigma \nabla^N_{\sigma}$.

We will take the orthonormal frame $\{e_\alpha\}$ at a point in
$\partial N$ so that $e_n$ is the inward-pointing unit normal vector
there. Let $\dvol_N$ denote the Riemannian density on $N$, and
similarly for $\dvol_{\partial N}$.
Given $\psi_1, \psi_2 \in C^\infty(N; E)$, we have
\begin{equation} \label{2.2}
\int_N \langle D^N \psi_1, \psi_2 \rangle \dvol_N -
\int_N \langle \psi_1, D^N \psi_2 \rangle \dvol_N  =
- i \int_{\partial N} \langle \psi_1, \gamma^n \psi_2 \rangle \dvol_N.
\end{equation}
The Lichnerowicz formula implies
\begin{equation} \label{2.3}
  (D^N)^2 = (\nabla^N)^* \nabla^N + \frac{R_N}{4} - \frac14 
  [\gamma^\sigma, \gamma^\tau] \left(
\frac18 \widehat{R}_{ab \sigma \tau} [\widehat{\gamma}^a,
  \widehat{\gamma}^b] \right).
\end{equation}

We now extend some computations in \cite[Proof of Lemma 4.1]{Lott (2000)}.
Suppose that $D^N \psi = 0$. Then (\ref{2.3}) implies that
\begin{align} \label{2.4}
0 = & 
  \int_N |\nabla^N \psi |^2 \: \dvol_N +
\int_{\partial N} \langle \psi, \nabla^N_{e_n} \psi \rangle
  + \frac14 \int_N R_N |\psi|^2 \: \dvol_N - \\
  & \frac{1}{32} \int_N \widehat{R}_{ab \sigma \tau} \langle \psi,
           [\gamma^\sigma, \gamma^\tau] [\widehat{\gamma}^a,
  \widehat{\gamma}^b] \psi \rangle. \notag
  \end{align}
Now $D^N \psi = 0$ implies that on $\partial N$, we have
\begin{align} \label{2.5}
  \nabla^N_{e_n} \psi  = & - \gamma^n \sum_{\mu = 1}^{n-1}
  \gamma^\mu \nabla^N_\mu \psi \\
   = &
  - \gamma^n \sum_{\mu = 1}^{n-1} \gamma^\mu
  \left( \nabla^{\partial N}_\mu \psi + \frac14 \omega_{n\beta \mu} \gamma^n
  \gamma^\beta \psi + \frac14 \widehat{\omega}_{mb\mu}
  \widehat{\gamma}^m \widehat{\gamma}^b \psi \right) \notag \\
  = & D^{\partial N} \psi + \frac{H_{\partial N}}{4} \psi \:  - \:
  \frac14 \gamma^n \gamma^\mu \widehat{\gamma}^m \widehat{\gamma}^b
  \widehat{A}_{b \mu} \psi, \notag
\end{align}
where
\begin{equation} \label{2.6}
D^{\partial N} =  - \gamma^n \sum_{\mu = 1}^{n-1} \gamma^\mu
   \nabla^{\partial N}_\mu
\end{equation}
is the Dirac operator on $\partial N$ coupled to
$(\partial f)^* S_M$, $\widehat{A}$ is the second fundamental form
of $M$ and $\widehat{A}_{b \mu} = \widehat{A}(\widehat{e}_b,
(\partial f)_*(e_\mu))$.
Hence
\begin{align} \label{2.7}
0 = & 
  \int_N |\nabla \psi |^2 \: \dvol_N 
  + \frac14 \int_N R_N |\psi|^2 \: \dvol_N -
  \frac{1}{32} \int_N \widehat{R}_{ab \sigma \tau}  \langle \psi,
          [\gamma^\sigma, \gamma^\tau] [\widehat{\gamma}^a,
            \widehat{\gamma}^b] \psi \rangle \dvol_N +  \\
          &
          \int_{\partial N} \langle \psi, D^{\partial N} \psi \rangle \dvol_{\partial N}
          +
          \frac14 \int_{\partial N} H_{\partial N}
          |\psi|^2 \dvol_{\partial N} - \frac14
          \int_{\partial N} \widehat{A}_{b \mu}
          \langle \psi, \gamma^n \gamma^\mu \widehat{\gamma}^m
          \widehat{\gamma}^b
\psi \rangle \dvol_{\partial N}. \notag
\end{align}

From \cite[Section 1.1]{Goette-Semmelmann (2002)}, 
\begin{equation} \label{2.8}
\frac{1}{32}  \widehat{R}_{ab \sigma \tau} 
          [\gamma^\sigma, \gamma^\tau] [\widehat{\gamma}^a,
            \widehat{\gamma}^b] \le \frac14 f^* R_M \Id_E.
  \end{equation}

\begin{lemma} \label{2.9} If $\widehat{A} \ge 0$ then
  \begin{equation} \label{2.10}
\widehat{A}_{b \mu}
           \gamma^n \gamma^\mu \widehat{\gamma}^m
          \widehat{\gamma}^b \le (\partial f)^* H_{\partial M} \Id_E.
\end{equation}
\end{lemma}
\begin{proof}
We use the method of proof of \cite[Section 1.1]{Goette-Semmelmann (2002)}.
  Put $\widehat{L} = \sqrt{\widehat{A}}$. Then
  \begin{equation} \label{2.11}
\widehat{A}_{b \mu}
           \gamma^n \gamma^\mu \widehat{\gamma}^m
           \widehat{\gamma}^b  =
\widehat{A}_{bc} \langle e_c, (\partial f)_*(e_\mu) \rangle
           \gamma^n \gamma^\mu \widehat{\gamma}^m
           \widehat{\gamma}^b=
           \widehat{L}_{ab}
           \widehat{L}_{ac} \langle e_c, (\partial f)_*(e_\mu) \rangle
           \gamma^n \gamma^\mu \widehat{\gamma}^m
           \widehat{\gamma}^b,
  \end{equation}
  so
\begin{align} \label{2.12}
&  \widehat{A}_{b \mu}
           \gamma^n \gamma^\mu \widehat{\gamma}^m
           \widehat{\gamma}^b  = \\
&           \frac12 \sum_a \left[
           \left(
           \widehat{L}_{ab} \widehat{\gamma}^m
           \widehat{\gamma}^b + \widehat{L}_{ac} \langle e_c, (\partial f)_*(e_\mu) \rangle
           \gamma^n \gamma^{\mu} 
           \right)^2 - \left(
           \widehat{L}_{ab} \widehat{\gamma}^m
           \widehat{\gamma}^b \right)^2 -
           \left( \widehat{L}_{ac} \langle e_c, (\partial f)_*(e_\mu) \rangle
           \gamma^n \gamma^{\mu} 
           \right)^2
           \right]. \notag
\end{align}
Now $\widehat{L}_{ab} \widehat{\gamma}^m
           \widehat{\gamma}^b + \widehat{L}_{ac} \langle e_c, (\partial f)_*(e_\mu) \rangle
           \gamma^n \gamma^{\mu}$ is skew-Hermitian, so has
           nonpositive square. Also,
           \begin{equation} \label{2.13}
             \left(
           \widehat{L}_{ab} \widehat{\gamma}^m
           \widehat{\gamma}^b \right)^2 = -
           \widehat{L}_{ab}^2 = - \widehat{A}_{bb} = - (\partial f)^*
           H_{\partial M}
           \end{equation}
           and
           \begin{align} \label{2.14}
             \left( \widehat{L}_{ac} \langle e_c, (\partial f)_*(e_\mu) \rangle
           \gamma^n \gamma^{\mu} 
           \right)^2 = & -
           \left( \widehat{L}_{ac} \langle e_c, (\partial f)_*(e_\mu) \rangle \right)^2
           = - \widehat{A}_{cd} \langle e_c, (\partial f)_*(e_\mu) \rangle
           \langle e_d, (\partial f)_*(e_\mu) \rangle \\
           = & -
           \widehat{A}((\partial f)_*(e_\mu),(\partial f)_*(e_\mu)) \ge -
           (\partial f)^* H_{\partial M}, \notag
           \end{align}
           using the fact that $\partial f$ is distance-nonincreasing.

           This proves the lemma.
  \end{proof}

\begin{proposition} \label{2.15}
  Let $N$ and $M$ be compact connected
  even dimensional Riemannian manifolds with boundary.
  Let $f \: : \: N \rightarrow M$ be a smooth spin map and let
  $\partial f \: : \: \partial N \rightarrow \partial M$
  denote the restriction to the boundary.
Suppose that
\begin{itemize}
\item $f$ is $\Lambda^2$-nonincreasing and $\partial f$ is
  distance-nonincreasing,
\item $M$ has nonnegative curvature operator and
    $\partial M$ has nonnegative second fundamental form,
    \item 
    $R_N \ge f^* R_M$ and
      $H_{\partial N} \ge (\partial f)^* H_{\partial M}$, and
      \item 
        There is a nonzero $\psi \in C^\infty(N; E)$ with
        $D^N \psi = 0$ on $N$ and
        $\int_{\partial N} \langle \psi, D^{\partial N} \psi \rangle \:
        \dvol_{\partial N} \ge 0$ on $\partial N$.
\end{itemize}
Then $R_N = f^* R_M$, $H_{\partial N} = (\partial f)^* H_{\partial M}$ and
$\psi$ is parallel.
\end{proposition}
\begin{proof}
This follows from (\ref{2.7}), (\ref{2.8}) and Lemma \ref{2.9}. 
\end{proof}

\begin{proposition} \label{2.16}
  Suppose that the assumptions of Proposition \ref{2.15} hold.
  \begin{enumerate}
\item    If $0 < \Ric_M < \frac12 R_M g_M$ then $f$ is a Riemannian submersion.
\item If $\Ric_M > 0$ and $f$ is distance-nonincreasing then $f$ is a
  Riemannian submersion
\item If $M$ is flat then $N$ is Ricci-flat.
  \end{enumerate}
\end{proposition}
\begin{proof}
  Part (1) follows from the computation in
  \cite[Section 1.2]{Goette-Semmelmann (2002)}. Part (2) follows from
  \cite[Remark 1.2]{Goette-Semmelmann (2002)}. For part (3), we know that
  $S_M$ has a flat unitary connection.  Around a point $p \in N$, we can
  write $\psi = \sum_{a} \psi^a s_a$, where $\{s_a\}$ is a parallel basis
  of $f^* S_M$ and $\psi^a$ is a local section of $S_N$.
  Then $\nabla \psi = 0$ implies that
  $R^N_{\alpha \beta \sigma \tau} [\gamma^\alpha, \gamma^\beta] \psi^a = 0$
  for each $a$. As some $\psi^a$ is nonzero, it follows from
\cite[Corollary 2.8]{BHMMM (2015)} that $N$ is Ricci-flat near $p$. 
  \end{proof}

\begin{remark} \label{2.17}
  Under the assumptions of Proposition \ref{2.15}, if $M$ is flat and spin then
  we can say more precisely that the universal cover of $N$ admits a
  nonzero parallel spinor field.
  \end{remark}

\subsection{Local boundary conditions}

We represent the generators of the Clifford algebra as
\begin{equation} \label{2.18}
  \gamma^n =
  \begin{pmatrix}
    0 & -i \\
    i & 0
    \end{pmatrix},
  \gamma^\mu =
  \begin{pmatrix}
    0 & \sigma^{\mu} \\
    \sigma^{\mu} & 0
    \end{pmatrix},
\end{equation}
where $\{\sigma^\mu \}_{\mu = 1}^{n-1}$ are generators for the
Clifford algebra on $\R^{n-1}$. Put
\begin{equation} \label{2.19}
\slashed{\partial}^{\partial N}
= - i \sum_{\mu = 1}^{n-1} \sigma^{\mu} \nabla^{\partial N}_\mu.
\end{equation}
With respect to the $\Z_2$-grading on $E$ coming from $S_N$, we have
\begin{equation} \label{2.20}
D^{\partial N} =    \begin{pmatrix}
    - \slashed{\partial}^{\partial N} & 0 \\
    0 & \slashed{\partial}^{\partial N}
    \end{pmatrix}. 
  \end{equation}

  As both
  $S_N$ and $S_M$ are $\Z_2$-graded, there is a total $\Z_2$-grading
  on $E$ and a bigrading
\begin{align} \label{2.21}
  E_+ & = E_{++} \oplus E_{--}, \\
  E_- & = E_{+-} \oplus E_{-+}. \notag
\end{align}
Given $\psi \in C^\infty(N; E)$, we decompose it as
$\psi = \psi_{++} + \psi_{--}
+ \psi_{+-} + \psi_{-+}$.
Define boundary conditions by
\begin{equation} \label{2.22}
  \psi_{++} = \psi_{--}, \: \: \: \psi_{+-} = \psi_{-+}.
  \end{equation}
Note that the boundary conditions do not mix $E_+$ and $E_-$.
\begin{lemma} \label{2.23}
  The operator $D^N$ is formally self-adjoint under the boundary
  conditions (\ref{2.22}).
  \end{lemma}
\begin{proof}
  Using (\ref{2.2}), it suffices to show that if
  $\psi^1, \psi^2 \in C^\infty(N; E)$ satisfy the boundary conditions then
  $\int_{\partial N} \langle \psi^1, \gamma^n \psi^2 \rangle \dvol_{\partial N}
  = 0$. For $i \in \{1,2\}$, let us write
  \begin{equation} \label{2.24}
    \psi^i = \psi^i_{++} + \psi^i_{+-} +
    \psi^i_{-+} + \psi^i_{--}.
  \end{equation}
  Then
  \begin{equation} \label{2.25}
    \langle \psi^1, \gamma^n \psi^2 \rangle =
    i \left( \langle \psi^1_{-+}, \psi^2_{++} \rangle +
    \langle \psi^1_{--}, \psi^2_{+-} \rangle \right) - i
    \left( \langle \psi^1_{++}, \psi^2_{-+} \rangle +
    \langle \psi^1_{+-}, \psi^2_{--} \rangle \right)
     = 0.
  \end{equation}
  This proves the lemma.
\end{proof}

One can check that (\ref{2.22}) gives elliptic boundary conditions in the
sense of \cite{Atiyah-Bott (1964)}. Hence there is a Fredholm operator
$D^N_+$, with domain consisting of the $H^1$-sections $\psi$ of
$E_+$ so that the boundary trace of $\psi_{++}$ equals the boundary trace
of $\psi_{--}$.
The domain of the adjoint $(D^N_+)^*$ 
consisting of the $H^1$-sections $\psi$ of
$E_-$ so that the boundary trace of $\psi_{+-}$ equals the boundary trace
of $\psi_{-+}$.
In theory one can
compute the index of $D^N_+$ using the procedure described in
\cite{Atiyah-Bott (1964)}.  However, in this case we can use a more
direct approach.
First, we can deform the metric on $N$ so that it is a product near
$\partial N$, without changing the metric on $\partial N$, and
similarly for $M$. Next, we can deform $f$, while fixing $\partial f$,
so that in a
product neighborhood
$[0, \delta) \times \partial N$ of $\partial N$, the
  map $f$ takes the form $f(t,x) = (t, (\partial f)(x))$. These
  deformations do not change the index.

Let us discuss spinors on the double of a manifold; c.f.
\cite[Section 4.4]{Freed-Hopkins (2016)}.  Suppose that
$N$ is spin and its metric is a Riemannian product near $\partial N$. Let $DN$
be the double of $N$.  As $DN$ is the boundary of $I \times N$, it inherits a
spin structure.
We can extend the structure group from $\Spin(n)$ to $\Pin_+(n)$.
The involution on $DN$ lifts to an involution $T$ on
sections of $S_N$, that commutes with the Dirac operator
$D^{DN}$. Writing a product neighborhood of $\partial N \subset DN$
as $(- \delta, \delta) \times \partial N$, the involution acts on
a spinor field by $(T\psi)(t,x) = i \epsilon \gamma^n v(-t, x)$, where
$\epsilon$ is the $\Z_2$-grading operator.
(In Minkowski space the time reversal operator involves complex
conjugation, but that is not the case here.) In terms of the
$\Z_2$-grading on $S_N$, we can write
$i \epsilon \gamma^n = \begin{pmatrix}
  0 & I \\
  I & 0
  \end{pmatrix}$.
Note that it anticommutes with $\epsilon$.
The $T$-invariant
$H^1$-regular spinors on $DN$ can then be identified with the
$H^1$-regular spinors $\psi$
on $N$ so that the boundary trace of $\psi_+ $ equals the boundary trace
of $\psi_-$. Note that these boundary conditions mix chiralities, so there
is not a well-defined index problem.  This is a reflection of the fact
that in general, Dirac-type operators on manifolds with boundary do not
admit local boundary conditions for index problems.

In our case, we can pass to doubles and
extend $f$ to a spin map $F : DN \rightarrow DM$. 
There is an involution $T$ on  $S_{DN} \otimes F^* S_{DM}$ that commutes
with the twisted Dirac operator $D^{DN}$. On a
neighborhood of $\partial N$, it
acts on sections of $S_{DN} \otimes F^* S_{DM}$ by
$(T\psi)(t,x) = - \epsilon \gamma^n \widehat{\epsilon} \widehat{\gamma}^m
\psi(-t, x)$. The $T$-invariant $H^1$-regular sections on $DN$
can be identified
with the $H^1$-regular sections of $E$, on $N$, that satisfy the
boundary conditions (\ref{2.22}). Because of the $\Z_2$-grading on
$E$, we do obtain local boundary conditions for an index problem.

Thus the index of $D^N_+$ is the same as the index of $D^{DN}_+$ when
acting on the $T$-invariant sections on the double.  We can think of
$DN/\Z_2$ as $N$ with an orbifold structure,
so we effectively have an index problem
on the orbifold.  From \cite{Kawasaki (1981)}, the index is
$\int_N \widehat{A}(N) \ch(f^* S_M)$.
(Since $\partial N$ is odd dimensional and the characteristic forms have
even degree, there is no boundary contribution.)
Here $S_M$ is $\Z_2$-graded, and
$\ch(S_M)$ equals the Euler form of $TM$.
From our present assumptions about
a product structure near the boundary, $\ch(S_M)$ vanishes near
$\partial M$ and represents
$\chi(M) [M, \partial M] \in \HH^m(M, \partial M; o_M)$. Hence the
index is $\chi(M) \int_N \widehat{A}(N) f^* [M, \partial M]$.

When combined with Propositions \ref{2.15} and \ref{2.16}, this proves Theorem
\ref{1.1}.

\section{Proof of Theorem \ref{1.3}}

In this section, we prove Theorem \ref{1.3}.  We begin more generally.
Let $N$ and $M$ be compact connected
  even dimensional Riemannian manifolds with boundary.
  Let $f \: : \: N \rightarrow M$ be a smooth spin map and let
  $\partial f \: : \: \partial N \rightarrow \partial M$
  denote the restriction to the boundary.
Suppose that
\begin{itemize}
\item $f$ is $\Lambda^2$-nonincreasing,
  \item $M$ has nonnegative curvature operator and
    $\partial N$ has vanishing mean curvature, and
    \item 
      $R_N \ge f^* R_M$.
      \end{itemize}

To avoid confusion with the previous section,
we now write the Dirac-type operator $D^N$ as
$\widetilde{D}^N$. 
Define a boundary Dirac-type operator by
\begin{equation} \label{3.1}
  \widetilde{D}^{\partial N} = - \gamma^n \sum_{\mu = 1}^{n-1}
  \gamma^\mu \nabla^N_{e_\mu} =   - \gamma^n \sum_{\mu = 1}^{n-1} \gamma^\mu
  \left( \nabla^{\partial N}_\mu + \frac14 \widehat{\omega}_{mb\mu}
  \widehat{\gamma}^m \widehat{\gamma}^b \right),
\end{equation}
where the last equality uses the fact that $H_{\partial N} = 0$; c.f.
equation (\ref{2.5}).
If $\widetilde{D}^N \psi = 0$ then (\ref{2.7}) becomes
\begin{align} \label{3.2}
0 = & 
  \int_N |\nabla^N \psi |^2 \: \dvol_N 
  + \frac14 \int_N R_N |\psi|^2 \: \dvol_N -
  \frac{1}{32} \int_N \widehat{R}_{ab \sigma \tau}  \langle \psi,
          [\gamma^\sigma, \gamma^\tau] [\widehat{\gamma}^a,
            \widehat{\gamma}^b] \psi \rangle \dvol_N +  \\
          &
          \int_{\partial N} \langle \psi, \widetilde{D}^{\partial N} \psi \rangle \dvol_{\partial N}. \notag
\end{align}

Hereafter we use the 
$\Z_2$-grading (\ref{2.21}).  In terms of it, we can write
\begin{equation} \label{3.3}
  \widetilde{D}^{\partial N} =
  \begin{pmatrix}
    - \widetilde{\slashed{\partial}}^{\partial N} & 0 \\
    0 & \widetilde{\slashed{\partial}}^{\partial N}
  \end{pmatrix}
\end{equation}
for an elliptic self-adjoint operator
$\widetilde{\slashed{\partial}}^{\partial N}$. (Here we use $\gamma^n$ to
implicitly identify $E_+ \Big|_{\partial N}$ and
$E_- \Big|_{\partial N}$.)
Let $P^{> 0}$ denote projection onto the subspace spanned by eigenvectors
of $\widetilde{\slashed{\partial}}^{\partial N}$ with positive eigenvalue, and
similarly for $P^{\le 0}$.
Given $\psi \in C^\infty(N; E)$,
let $\psi^{\partial N}_\pm$ be the components of its boundary restriction,
relative to the $\Z_2$-grading. 
We impose the boundary conditions 
\begin{equation} \label{3.4}
  P^{>0} \psi^{\partial N}_+ =   P^{\le 0} \psi^{\partial N}_- = 0. 
\end{equation}

\begin{lemma} \label{3.5}
  The operator $\widetilde{D}^N$ is formally self-adjoint under the boundary
  conditions (\ref{3.4}).
  \end{lemma}
\begin{proof}
  Using (\ref{2.2}), it suffices to show that if
  $\psi^1, \psi^2 \in C^\infty(N; E)$ satisfy the boundary conditions then
  $\int_{\partial N} \langle \psi^1, \gamma^n \psi^2 \rangle \dvol_{\partial N}
  = 0$. In terms of the $\Z_2$-grading on $E$, we can write
  $\gamma^n =
  \begin{pmatrix}
    0 & -i \\
    i & 0
    \end{pmatrix}$. Then
  \begin{equation} \label{3.6}
    \langle \psi^1, \gamma^n \psi^2 \rangle =
    i \langle \psi^1_-, \psi^2_+ \rangle - i \langle \psi^1_+, \psi^2_-
    \rangle.
    \end{equation}
  The boundary conditions imply that
  \begin{equation} \label{3.7}
    \int_{\partial N} \langle \psi^1_-, \psi^2_+ \rangle \dvol_{\partial N} =
    \int_{\partial N} \langle \psi^1_+, \psi^2_-
    \rangle \dvol_{\partial N} = 0.
    \end{equation}
  This proves the lemma.
  \end{proof}
  
The boundary conditions (\ref{3.4}) make the differential operator
$\widetilde{D}^N$ into an
elliptic self-adjoint operator, which we write as
$\widetilde{D}^{N, APS}$.  Its domain is the set of
$H^1$-regular sections $\psi$ of $E$ whose boundary trace
satisfies (\ref{3.4}).
The conditions (\ref{3.4})
differ slightly from the Atiyah-Patodi-Singer boundary conditions
\cite{Atiyah-Patodi-Singer (1975)}, which would be
$P^{\ge 0} \psi^{\partial X}_+ =   P^{<0} \psi^{\partial X}_- = 0$,
but the boundary conditions (\ref{3.4}) work just as well.

From (\ref{3.3}),
the boundary conditions imply that $\int_{\partial N} \langle
\psi, \widetilde{D}^{\partial N} \psi \rangle \dvol_{\partial N} \ge 0$.
We conclude from (\ref{3.2}) that if $\psi \in C^\infty(N; E)$ is a nonzero
solution of $\widetilde{D}^N \psi = 0$, satisfying the boundary conditions
(\ref{3.4}), then $R_N = f^* R_M$ and $\psi$ is parallel.

\begin{lemma} \label{3.8}
If $R_N \neq f^* R_M$ then the kernel of $\widetilde{D}^N$ vanishes.

Suppose that $R_N = f^* R_M$. 
The kernel of $\widetilde{D}^{N,APS}_+$ is isomorphic to the vector space of
parallel sections of $E_+$.
The kernel of $\widetilde{D}^{N,APS}_-$ vanishes.
  \end{lemma}
\begin{proof}
  We have already proved the first statement of the lemma.  For the
  second statement, 
  by elliptic regularity, an element of the kernel is smooth on $N$.
  Suppose that $R_N = f^* R_M$. If $\widetilde{D}^{N,APS} \psi = 0$ then
  $\psi$ is parallel.  From the definition of $\widetilde{D}^{\partial N}$,
  it follows that $\widetilde{D}^{\partial N} \psi = 0$. Writing
  $\psi = \psi_+ + \psi_-$, the boundary conditions (\ref{3.4}) imply that
  $\psi_- = 0$. On the other hand, if $\psi_+$ is a parallel section
  of $E_+$ then $\widetilde{D}^N \psi_+ = 0$ and $\psi_+$ satisfies the
  boundary condition (\ref{3.4}). Thus $\psi_+$ is in the kernel of
  $\widetilde{D}^{N,APS}_+$.
\end{proof}

Using Lemma \ref{3.8}, the task now is to find situations which guarantee that
$\widetilde{D}^{N,APS}_+$ has a nonzero kernel or, equivalently for us, that
it has a nonzero index. One situation that is easy to analyze is when
$N = M$ and $f$ is the identity map.  Then $E$ is isomorphic to
$\Lambda^*(TM)$, with the $\Z_2$-grading distinguishing even and odd forms.
As the constant function is always a nonzero parallel section of $E$, it
lies in the kernel of $\widetilde{D}^{N,APS}_+$.

To motivate the assumptions of Theorem \ref{1.3},
if we deform the Riemannian metric on $N$, allowing the boundary metric to
also change, then the index of $\widetilde{D}^{N,APS}_+$ can change.  The
change is determined by the spectral flow of
$\widetilde{\slashed{\partial}}^{\partial N}$, which could be hard to
compute.  However, if $\widetilde{\slashed{\partial}}^{\partial N}$
doesn't change in the deformation then
the index doesn't change.  Such is the case when 
the metric and the second fundamental form of $\partial N$
do not change in the deformation.

To return to the proof of Theorem \ref{1.3}, let us write $N^\prime = M$
for the case when $N$ is the same as $M$, with $f^\prime : N^\prime
\rightarrow M$ being the identity map.
Let $N$ and $f$ be as in the statement
of Theorem \ref{1.3}.
We can assume that $\partial f = \partial f^\prime$.
By the Atiyah-Patodi-Singer index theorem
\cite{Atiyah-Patodi-Singer (1975)}
and its extension to the case of a nonproduct structure near the boundary
\cite{Gilkey (1993)}, since the boundary data is the same
for $N$ and $N^\prime$,
the difference $\Delta$ of the indices is
\begin{equation} \label{3.9}
\Delta =   \int_N \widehat{A}(TN) f^* \ch(S_M) -
\int_{N^\prime} \widehat{A}(TN^\prime) (f^\prime)^* \ch(S_M).
\end{equation}
Here $\ch(S_M)$ denotes the explicit Chern form.  It has top
degree and equals the Euler form $e(M)$. Hence 
\begin{equation} \label{3.10}
  \Delta = (\deg(f) - \deg(f^\prime)) \int_M e(M)
  \end{equation}
Since $\partial f$ and $\partial f^\prime$ are spin diffeomorphisms, it
follows that $\deg(f) = \deg(f^\prime) = 1$.
Thus
$\Delta = 0$ and
$\widetilde{D}^{N,APS}_+$ has a positive index.

This proves the first part of Theorem \ref{1.3}. The second part
follows from Proposition \ref{2.16}.

\end{document}